\documentclass[oneside,english]{amsart}
\usepackage[T1]{fontenc}
\usepackage[latin9]{inputenc}
\usepackage{geometry}
\usepackage{amsthm}

\makeatletter
\numberwithin{equation}{section}
\numberwithin{figure}{section}
\theoremstyle{plain}
\newtheorem{thm}{\protect\theoremname}
\theoremstyle{plain}
\newtheorem{conj}[thm]{\protect\conjecturename}
\theoremstyle{plain}
\newtheorem{prob}[thm]{\protect\problemname}
\theoremstyle{plain}
\newtheorem{lem}[thm]{\protect\lemmaname}
\theoremstyle{plain}
\newtheorem{prop}[thm]{\protect\propositionname}
\theoremstyle{plain}
\newtheorem{cor}[thm]{\protect\corollaryname}

\makeatother

\usepackage{babel}
\providecommand{\conjecturename}{Conjecture}
\providecommand{\lemmaname}{Lemma}
\providecommand{\theoremname}{Theorem}
\providecommand{\problemname}{Problem}
\providecommand{\propositionname}{Proposition}
\providecommand{\corollaryname}{Corollary}
\begin{document}


\title{Settling Some Sum Suppositions}

\author{Tanay Wakhare$^\ast$ and Christophe Vignat$^\dag$}

\thanks{{\scriptsize
\hskip -0.4 true cm MSC(2010): Primary: 11A63; Secondary: 05A18.
}}

\address{$^\ast$~University of Maryland, College Park, MD 20742, USA}
\email{twakhare@gmail.com}
\address{$^\dag$~Tulane University, New Orleans, LA 70118, USA and Universit\'{e} Paris Sud, France}
\email{cvignat@tulane.edu, christophe.vignat@u-psud.fr}

\maketitle

\begin{abstract}
We solve multiple conjectures by Byszewski and Ulas about the sum of base $b$ digits function.
In order to do this, we develop general results about summations over
the sum of digits function. As a corollary, we describe an unexpected new result about the Prouhet-Tarry-Escott problem.
In some cases, this allows us to partition fewer than $b^N$ values into $b$ sets $\{S_1,\ldots,S_b\}$, such that 
$$\sum_{s\in S_1}s^k = \sum_{s\in S_2}s^k = \cdots = \sum_{s\in S_b}s^k $$
for $0\leq k \leq N-1$. The classical construction can only partition $b^N$ values such that the first $N$ powers agree. Our results are amenable to a computational search, which may discover new, smaller, solutions to this classical problem.
\end{abstract}

\section{Introduction}

The sum of base $b$ digits function $s_b(n)$ is ubiquitous in number theory and combinatorics; it has applications ranging from partitioning a set into equal valued subsets (the classical Prouhet-Tarry-Escott problem) \cite{PTE} to the modeling of quasicrystals \cite{Allouche2}. Developing the theory of summations over $s_b(n)$ and the related sequence $(-1)^{s_b(n)}$ therefore has diverse applications in pure mathematics. In particular, the classical solution to the Prouhet-Tarry-Escott problem is equivalent to proving that a certain sum over $(-1)^{s_2(n)}$ is zero, which highlights the power of this pure mathematical approach.

In this paper, using a variety of methods, we  
solve and generalize several conjectures proposed by Byszewski and Ulas \cite{Ulas}. 
As an intermediate step, we provide several general theorems about multiple summations involving the sequence $s_b(n)$, which may have 
other potential applications to questions about digit sums. While the questions we answer have a very special form, this paper highlights the usefulness of our methods, with potential applications to many other problems.

In particular, we prove and generalize the following two conjectures:
\begin{conj}\label{conj42}\cite[Conjecture 4.2]{Ulas}
For $r\ge1$ and $N_{1},\dots N_{r}$ positive integers, we have
\[
\sum_{n_{1}=0}^{2^{N_{1}}-1}\dots\sum_{n_{r}=0}^{2^{N_{r}}-1}\left(-1\right)^{\sum_{j=1}^{r}s_{2}\left(n_{j}\right)}\left(x+\sum_{j=1}^{r}n_{j}y_{j}\right)^{\sum_{j=1}^{r}N_{j}}=\left(-1\right)^{\sum_{j=1}^{r}N_{r}}2^{\sum_{j=1}^{r}\frac{N_{j}\left(N_{j}-1\right)}{2}}\left(\prod_{j=1}^{r}y_{j}^{N_{j}}\right)\left(\sum_{j=1}^{r}N_{j}\right)!.
\]
\end{conj}
\begin{conj}\label{conj44}\cite[Conjecture 4.4]{Ulas}
For $r\ge1$ and $N_{1},\dots N_{r}$ positive integers, we have
\begin{equation}
\label{eq:conjecture2}
\sum_{n_{1}=0}^{2^{N_{1}}-1}\dots\sum_{n_{r}=0}^{2^{N_{r}}-1}\left(-1\right)^{\sum_{j=1}^{r}s_{2}\left(n_{j}\right)}\left(\sum_{j=1}^{r}s_2(n_j)x_j + n_{j}y_{j}\right)^{\sum_{j=1}^{r}N_{j}}
=\left(-1\right)^{\sum_{j=1}^{r}N_{r}}\left(\sum_{j=1}^{r}N_{j}\right)! \prod_{j=1}^r \prod_{i_j=0}^{N_j-1}(x_j+2^{i_j}y_j).
\end{equation}
\end{conj}
Moreover, we show new results about the following problem:
\begin{prob}\cite[Problem 4.5]{Ulas}
With $\mathbf{x}=\left( x_{1},\ldots ,x_{m} \right)$, study the following family of polynomials
\begin{equation}
\label{eq:pb3}
H_{m,N}\left( t,\mathbf{x} \right)=\sum_{i_{1}=0}^{2^{N}-1}\dots\sum_{i_{m}=0}^{2^{N}-1}\left(-1\right)^{s_{2}\left(\sum_{j=1}^{m}i_{j}\right)}\left(t+\sum_{j=1}^{m}i_{j}x_{j}\right)^{N}.
\end{equation}
\end{prob}
For what follows, the $\beta_{k}^{(N)}$ weights are complex constants which we specify later, and $\Delta_k$ is the forwards difference operator in the variable $k$ defined by the action
\[
\Delta_k f\left( x+k \right)= f\left( x+k+1 \right) - f\left( x+k \right).
\]
Our basic tool is the generalization to an arbitrary base $b\in \mathbb{N}$ of identity \eqref{eq:identity} below \cite{Wakhare}, which allows us to transform sums involving the sequence $(-1)^{s_{2}\left(n\right)}$ into sums that involve iterated finite differences. 
\begin{thm}\label{betathm}
Let the constants $\beta_k^{(N)}$ be defined by \eqref{betaweights} and let $\xi$ be a $b$-th root of unity. Then, for an arbitrary function $f$,
\[
\sum_{n=0}^{b^{N}-1}\xi^{s_{b}\left(n\right)}f\left(x+ny\right)=\left(-1\right)^{N}\sum_{k=0}^{b^{N}-N-1}\beta_{k}^{\left(N-1\right)}\Delta_{k}^{N}f\left(x+ky\right).
\]
\end{thm}

It is worth noting that by taking $f$ to be a polynomial in $x$ of degree $<N$ in this theorem, we recover the classical solution to the Prouhet-Tarry-Escott problem. We recall that the Tarry-Escott problem asks for $b$ equally sized sets $\{S_1,S_2,\ldots,S_b\}$ such that the sums $\sum_{s\in S_i}s^k$ are all equal, for some given values of $k$. Prouhet's contribution was to describe $S_i$ when we begin with the $b^N$ elements $\{0,1,\ldots,b^N-1\}$ and $k<N$, by defining $S_i = \{l: 0\leq l \leq  b^N -1, s_b(l) \equiv i \pmod b\}$. In fact, we can show the slightly more general result $$\sum_{n=0}^{b^N-1}\xi^{s_b(n)} f(ny) = 0$$ for \textit{any} polynomial of degree $< N$. 
This classical result follows from Theorem \ref{betathm} since the $\Delta_k$ operator is degree lowering in $x$, so that the right-hand side vanishes for polynomials of degree $<N$. Note that \textit{a priori} this does not give us a rule of partition any $b^N$ integers, since in general we can only interpolate these values with a degree $b^N-1$ polynomial. In the special case that we can instead write $b^N$ as the values of a polynomial of degree $<N$ sampled at points in an arithmetic progression, we obtain Prouhet's result. However, this also has an important limitation: given $b^N$ points, we only have a rule to form sets whose $N-1$-th powers agree.

By inspecting the $r=1$ case of our Theorem \ref{thm12}, a generalization of Conjecture \ref{conj44}, we discover we have shown the following generalization of the Prouhet-Tarry-Escott result:
\begin{thm} (Generalized Prouhet-Tarry-Escott) Let $f$ be any polynomial of degree $<N$. Then
$$\sum_{n=0}^{b^N-1}\xi^{s_b(n)} f\left(s_2(n)x+ny\right) = 0.$$
\end{thm}
This result is unexpected because the digit sum function $s_2(n)$ is much less well behaved than the linear $ny$ term, and twists the function values we sample in an unexpected way. It gives us a deterministic noise term when sampling our points from an arithmetic progression, which can nevertheless be arbitrarily scaled since $y\in \mathbb{R}$ is free.

For example, taking the simple case $b=2, N=3, x=y=1$ gives us the partition $\{0,5,7,8\}\cup \{2,3,5,10\}$, where the sets are divided by the value of $s_2(n)\pmod 2$ and the values inside the sets are $n+s_2(n)$ for $0\leq n \leq 7$. We manually verify
\begin{align*}
0+5+7+8 =  2&0 = 2+3+5+10, \\
 0^2+5^2+7^2+8^2 =  1&38 = 2^2+3^2+5^2+10^2.
\end{align*}
Note that $5$ can be cancelled from both sides, so that we can recover the \textit{even smaller} six element partition $\{0,7,8\}\cup \{2,3,10\}$, which passes the classical barrier inherent to Prouhet's contruction. Taking $x, y \in \mathbb{Q}$ then generates even more exotic solutions. 

Though substantial generalizations of Prouhet's original result have been hard to come by, this work is interesting in that it provides ways to partition integer sets containing elements of multiplicity greater than one, so that we still retain nice equi-summability properties. Sorting elements of the same value into different sets also allows us to perform some cancellation, which then recovers unexpected new solutions. A computational search may allow us to discover new smaller solutions to the Prouhet-Tarry-Escott problem, by choosing $x$ and $y$ to maximize cancellation. This cancellation property is an unexpected new development in the study of the Tarry-Escott problem, which has spanned more than two centuries, and certainly merits independent study.



\section{Generalized Results}
To attack Conjecture \ref{conj42} we require a multiple version of a recent identity concerning sums over $(-1)^{s_2(n)}$. 
\begin{thm}\label{mainthm} \cite[Theorem 18]{Wakhare}
For integer $N$ and an arbitrary function $f$, we have
\begin{equation}
\sum_{n=0}^{2^{N}-1}\left(-1\right)^{s_{2}\left(n\right)}f\left(x+n\right)=\left(-1\right)^{N}\sum_{k=0}^{2^{N}-N-1}\alpha_{k}^{\left(N-1\right)}\Delta^{N}f\left(x+k\right)\label{eq:identity}
\end{equation}
\end{thm}
Here, we let $\Delta$ denote the forward finite difference operator with action $\Delta f(n) = f(n+1)-f(n)$. In what follows, we will heavily depend on $\alpha_k^{(N)}$, also known as sequence A131823 in the Online Encyclopedia of Integer Sequences (OEIS). For a fixed $N,$ we define $2^{N+1}-N-2$ coefficients $\alpha_{k}^{\left(N\right)}$ through the generating product 

\begin{equation}
\sum_{k=0}^{2^{N+1}-N-2}\alpha_{k}^{\left(N\right)}z^{k}=\prod_{l=0}^{N-1}\left(1+z^{2^{l}}\right)^{N-l}.\label{eq:gf}
\end{equation}

Alternatively, this sequence can be defined combinatorially as the number of integer points in the intersection of the parallepiped
$\left[0,1 \right] \times \left[0,3 \right] \times \cdots \times \left[0,2^N-1 \right]$ with the hyperplane $x_1+\dots +x_N=k$:
\[
\alpha_{k}^{\left(N\right)}=\#\left\{ 0\le k_{1}\le1,\ldots,0\le k_{N}\le2^{N}-1\thinspace\vert\thinspace k_{1}+\cdots+k_{N}=k\right\}.
\]

We construct a two-fold generalization of this theorem, first to arbitrary base, and then to scaled finite differences. To do this, we fix two positive integers $N$ and $b$, and let $\xi$ be a $b$-th root of unity. We can then define $b^{N+1}-N-2$ coefficients $\beta_{k}^{\left(N\right)}$ through the generating function
\begin{equation}
\sum_{k=0}^{b^{N+1}-N-2}\beta_{k}^{\left(N\right)}z^{k}=\prod_{l=0}^{N} \frac{1-z^{b^l}}{1-z}\left( 1 + (1+\xi)z^{b^l}+ (1+\xi+\xi^2 )z^{2 b^l} + \cdots + \left( \sum_{k=0}^{b-1} \xi^k \right)z^{(b-1) b^l}\right).   \label{betaweights}
\end{equation}
This representation can be significantly simplified -- each summand can be summed as a geometric series, or effectively telescoped, but this form most clearly highlights the polynomial nature of our product. Firstly, note that it in fact defines a polynomial in $z$, since $\frac{1-z^{b^l}}{1-z}$ is a finite geometric series of degree $b^l-1$. Additionally, the final partial sum over roots of unit vanishes since $  \sum_{k=0}^{b-1} \xi^k =0$, so that the degree of each term inside the product is $(b-1)b^l -1$, and the final polynomial does indeed have degree $b^{N+1}-N-2$. Next, the $\beta$ weights are in general complex valued and do not have a nice combinatorial interpretation; we lose the convenient property that, in base $b=2,$ the $\alpha$ weights count points on certain restricted hyperplanes. However, since none of our theorems depend on combinatorial properties of these coefficients, we can proceed freely. While this generating product does not give us any intuition for the $\beta$ weights, in practice we often only encounter sums such as $\sum_{k=0}^{b^{N+1}-N-2}\beta_n^{(N)}$ which can be easily derived from the generating product for the $\beta$ weights.

\begin{lem}
\label{lemma6}
The first two moments of the $\beta$ weights are
\begin{equation}
\label{sumbetak}
\sum_{k=0}^{b^{N}-N-1} \beta_{k}^{\left( N -1 \right)} = \frac{b^{\frac{N\left( N+1 \right)}{2}}}{\left( 1-\xi \right)^N}
\end{equation}
and
\begin{equation}
\label{sumkbetak}
\sum_{k=0}^{b^{N}-N-1} k\beta_{k}^{\left( N -1 \right)} = \frac{b^{\frac{N\left( N+1 \right)}{2}}}{\left( 1-\xi \right)^N} \left[\frac{1-b^N}{1-b}\left( \frac{b}{2} + \frac{\xi}{1-\xi} \right) - \frac{N}{2} \right].
\end{equation}

\end{lem}

\begin{proof}

Let 
\[
F\left(z\right)=\sum_{k=0}^{b^{N}-N-1}\beta_{k}^{\left(N-1\right)}z^{k}=\prod_{l=0}^{N-1}\left(\sum_{k=0}^{b^{l}-1}z^{k}\right)\left(\sum_{m=0}^{b-1}\frac{1-\xi^{k+1}}{1-\xi}z^{kb^{l}}\right)
\]
denote the generating function of the coefficients $\beta_{k}^{\left(N-1\right)}.$
We compute
\begin{align*}
F\left(1\right) & =\prod_{l=0}^{N-1}b^{l}\sum_{m=0}^{b-1}\frac{1-\xi^{k+1}}{1-\xi}=\prod_{l=0}^{N-1}b^{l}\left(\frac{b}{1-\xi}-\frac{\xi}{1-\xi}\frac{1-\xi^{b}}{1-\xi}\right) =\prod_{l=0}^{N-1}\frac{b^{l+1}}{1-\xi},
\end{align*}
where we have used the fact that $\xi^{b}=1.$ We then deduce
\[
\sum_{k=0}^{b^{N}-N-1}\beta_{k}^{\left(N-1\right)}=\frac{b^{\frac{N\left(N+1\right)}{2}}}{\left(1-\xi\right)^{N}}.
\]
Next we compute
\[
\frac{F'\left(z\right)}{F\left(z\right)}=\sum_{l=0}^{N-1}\frac{\frac{d}{dz}\left(\sum_{k=0}^{b^{l}-1}z^{k}\right)}{\sum_{k=0}^{b^{l}-1}z^{k}}+\sum_{l=0}^{N-1}\frac{\frac{d}{dz}\left(\sum_{m=0}^{b-1}\frac{1-\xi^{k+1}}{1-\xi}z^{kb^{l}}\right)}{\sum_{m=0}^{b-1}\frac{1-\xi^{k+1}}{1-\xi}z^{kb^{l}}}.
\]
The first term evaluated at $z=1$ is
\[
\sum_{l=0}^{N-1}\frac{\sum_{k=0}^{b^{l}-1}k}{\sum_{k=0}^{b^{l}-1}1}=\sum_{l=0}^{N-1}\frac{\frac{b^{l}\left(b^{l}-1\right)}{2}}{b^{l}}=\frac{1}{2}\sum_{l=0}^{N-1}\left(b^{l}-1\right)=\frac{1}{2}\left[\frac{1-b^{N}}{1-b}-N\right].
\]
The second term evaluated at $z=1$ is
\[
\sum_{l=0}^{N-1}\frac{\sum_{m=0}^{b-1}kb^{l}\frac{1-\xi^{k+1}}{1-\xi}}{\sum_{m=0}^{b-1}\frac{1-\xi^{k+1}}{1-\xi}}=\sum_{l=0}^{N-1}b^{l}\frac{\sum_{m=0}^{b-1}k\frac{1-\xi^{k+1}}{1-\xi}}{\sum_{m=0}^{b-1}\frac{1-\xi^{k+1}}{1-\xi}};
\]
the numerator and denominator are respectively evaluated as
\[
\sum_{m=0}^{b-1}k\frac{1-\xi^{k+1}}{1-\xi}=\frac{1}{1-\xi}\left[\frac{b\left(b-1\right)}{2}+b\frac{\xi}{1-\xi}\right]
\]
and, since $\xi$ is a $b-$th root of unity,
\[
\sum_{m=0}^{b-1}\frac{1-\xi^{k+1}}{1-\xi}=\frac{b}{1-\xi}.
\]
Therefore, the second term is 
\[
\left(\frac{b-1}{2}+\frac{\xi}{1-\xi}\right)\left(\frac{1-b^{N}}{1-b}\right)
\]
and we deduce
\begin{align*}
\sum_{k=0}^{b^{N}-N-1} k\beta_{k}^{\left( N -1 \right)} = F'\left(1\right) & =F\left(1\right)\left[\frac{1}{2}\frac{1-b^{N}}{1-b}-\frac{N}{2}+\frac{1-b^{N}}{1-b}\left(\frac{b-1}{2}+\frac{\xi}{1-\xi}\right)\right]\\
 & =\frac{b^{\frac{N\left(N+1\right)}{2}}}{\left(1-\xi\right)^{N}}\left[\frac{1-b^{N}}{1-b}\left(\frac{b}{2}+\frac{\xi}{1-\xi}\right)-\frac{N}{2}\right].
\end{align*}
\end{proof}

We also have the following simple convolution, which lies at the center of the arguments that follow.
\begin{lem}\label{betaconv}
We have the two convolutions
\begin{align}
\beta_n^{(N)} = \sum_{k=0}^n \binom{n-k+N}{N}\xi^{s_b(k)},\,\,  0\le n \le b^N-1, \label{dual1} \\
\xi^{s_b(n)}  = \sum_{k=0}^{n} \binom{N}{k}(-1)^k \beta_{n-k}^{(N-1)}, \,\,0\le n \le b^N-1 . \label{dual2}
\end{align}
\end{lem}
\begin{proof}
Both of these results follow from elementary product identities; we first need to show
\begin{equation} \label{conv1}
\sum_{k=0}^{b^{N}-1}\xi^{s_b(k)}z^{k} =   (1-z)^N\sum_{k=0}^{b^{N}-N-1}\beta_{k}^{\left(N-1\right)}z^{k}.
\end{equation}
This follows by manipulating the generating products for both series. Recall that 
\begin{align*}
(1-z)^{N+1}\sum_{k=0}^{b^{N+1}-N-2}\beta_{k}^{\left(N\right)}z^{k} &= (1-z)^{N+1}\prod_{l=0}^{N} \frac{1-z^{b^l}}{1-z}\left( 1 + (1+\xi)z^{b^l}+ (1+\xi+\xi^2 )z^{2 b^l} + \cdots + \left( \sum_{k=0}^{b-1} \xi^k \right)z^{(b-1) b^l}\right) \\
&=\prod_{l=0}^{N}\left( {1-z^{b^l}}\right)\left( 1 + (1+\xi)z^{b^l}+ (1+\xi+\xi^2 )z^{2 b^l} + \cdots + \left( \sum_{k=0}^{b-1} \xi^k \right)z^{(b-1) b^l}\right) \\
&= \prod_{l=0}^{N-1} \left(1+\xi z^{b^l} + \xi^2 z^{2\cdot b^l}+\cdots + \xi^{b-1} z^{(b-1)b^l}\right) \\
&=\sum_{k=0}^{b^{N}-1}\xi^{s_b(k)}z^{k}.
\end{align*}
The final step involved telesoping the summand. We can then expand $(1-z)^N$ and compare coefficients of $z^n$: 
\begin{align*}
\sum_{n=0}^{b^N-1}\xi^{s_b(n)}z^n &= (1-z)^N\sum_{k=0}^{b^{N}-N-1}\beta_{k}^{\left(N-1\right)}z^{k} \\
&=  \left( \sum_{k=0}^N \binom{N}{k} (-1)^k x^k \right) \sum_{k=0}^{b^{N}-N-1}\beta_{k}^{\left(N-1\right)}z^{k} \\
&= \sum_{n=0}^{b^N-1} z^n  \sum_{k=0}^{n} \binom{N}{k}(-1)^k \beta_{n-k}^{(N-1)}.
\end{align*}
The other convolution similarly follows from instead considering
$$\sum_{k=0}^{b^{N+1}-N-2}\beta_{k}^{\left(N\right)}z^{k} =  \frac{1}{(1-z)^{N+1}} \sum_{k=0}^{b^{N}-1}\xi_{k}^{s_b(k)}z^{k},$$
expanding $$\frac{1}{(1-z)^{N+1}} = \left(\sum_{k=0}^\infty \binom{k+N}{N}z^k\right),$$
and then comparing coefficients of $z^n$.  
\end{proof}

We can now prove one of our main results, a generalization of \cite[Theorem 18]{Wakhare} to arbitrary base and to asymmetric functional arguments.
\begin{thm}
\label{Thm7}
Let the sequence $\left \{\beta_k^{(N)} \right \}$ be defined by \eqref{betaweights} and let $\xi$ be a $b$-th root of unity. Then for an arbitrary function $f$,
\[
\sum_{n=0}^{b^{N}-1}\xi^{s_{b}\left(n\right)}f\left(x+ny\right)=\left(-1\right)^{N}\sum_{k=0}^{b^{N}-N-1}\beta_{k}^{\left(N-1\right)}\Delta_{k}^{N}f\left(x+ky\right),
\]
where $\Delta_{k}f\left(x+ky\right)=f\left(x+\left(k+1\right)y\right)-f\left(x+ky\right)$
is the forward finite difference operator in the variable $k$  \footnote{Theorem \ref{mainthm} is symmetric in $x$ and $k$ so the $\Delta$
operator can act on either of them. However, we now deal with functions
without this summetry, so that we need to specify which
variable the $\Delta$ operator acts on}.
\end{thm}
\begin{proof}
On a high level, we are expanding the finite difference operator as a sum, and then sfhiting it over to the $\beta$ weights while exploiting the convolution from Lemma \ref{betaconv}. We have the following direct argument, noting that for $l < 0$ we have $\beta^{(N)}_{l} = 0$:
\begin{align*}
\sum_{n=0}^{b^N-1} \xi^{s_b(n)} f(x+ny) &= \sum_{n=0}^{b^N-1} f(x+ny) \sum_{k=0}^{n} \binom{N}{k}(-1)^k \beta_{n-k}^{(N-1)} \\
&=  \sum_{k=0}^{b^N-1} \sum_{n=k}^{b^N-1} f(x+ny) \binom{N}{k}(-1)^k \beta_{n-k}^{(N-1)} \\
&=  \sum_{k=0}^{b^N-1} \sum_{n=0}^{b^N-k-1} f\left(x+(n+k)y\right) \binom{N}{k}(-1)^k \beta_{n}^{(N-1)} \\
&=  \sum_{k=0}^{N} \sum_{n=0}^{b^N-N-1}  f\left(x+(n+k)y\right)  \binom{N}{k}(-1)^k \beta_{n}^{(N-1)} \\
&=   \sum_{n=0}^{b^N-N-1}\beta_{n}^{(N-1)}\sum_{k=0}^{N}  f\left(x+(n+k)y\right)  \binom{N}{k}(-1)^k  \\
&=  (-1)^N \sum_{n=0}^{b^N-N-1} \beta_n^{(N-1)}\Delta^N_n f(x+ny).
\end{align*}
\end{proof}

We can compute some simple examples, which generalize \cite[p 115-116]{Allouche1} from base $2$ to an arbitary base, while also inserting a free variable $y$. 
\begin{cor}
We have the sums
\[
\sum_{n=0}^{b^{N}-1}\xi^{s_{b}\left(n\right)}\left(x+ny\right)^{N}=
\left(-1\right)^{N}N!y^{N}\frac{b^{\frac{N\left(N+1\right)}{2}}}{\left(1-\xi\right)^{N}}
\]
and
\[
\sum_{n=0}^{b^{N}-1}\xi^{s_{b}\left(n\right)}\left(x+ny\right)^{N+1}=
\left(-1\right)^{N}\left(N+1\right)!y^{N}\frac{b^{\frac{N\left(N+1\right)}{2}}}{\left(1-\xi\right)^{N}}\left[x+y\left(\frac{1-b^{N}}{1-b}\left(\frac{b}{2}+\frac{\xi}{1-\xi}\right)\right)\right].
\]
\end{cor}
\begin{proof}
The first result follows by taking $f(x)=x^N$ and the second from $f(x)=x^{N+1}$ in Thm \ref{Thm7}. We also need the two following results that can be easily checked by induction on $N$:
\[
\Delta_{k}^{N}\left(x+ky\right)^{N}=N!y^{N}
\]
and
\[
\Delta_{k}^{N}\left(x+ky\right)^{N+1}=\left(N+1\right)!y^{N}\left[x+\frac{N}{2}y+ky\right].
\]
Elementary algebra using \eqref{sumbetak} and \eqref{sumkbetak} in Lemma \ref{lemma6} yields the results.
\end{proof}


A multiple sum version of Thm. \ref{Thm7}  can be stated as follows:

\begin{thm}\label{thm2}
A multiple summation version of 
Theorem \ref{Thm7}, where the $N_i$ are positive integers, is as follows: 
\begin{align}
\sum_{n_{1}=0}^{b^{N_{1}}-1}\dots\sum_{n_{r}=0}^{b^{N_{r}}-1}\xi^{\sum_{j=1}^{r}s_{b}\left(n_{j}\right)}f\left(x+\sum_{j=1}^{r}n_{j}y_{j}\right) & =\left(-1\right)^{\sum_{j=1}^{r}N_{j}}\sum_{k_{1}=0}^{b^{N_{1}}-N_{1}-1}\dots\sum_{k_{r}=0}^{b^{N_{r}}-N_{r}-1}\label{eq:tensor}\\
 &\times \left(\prod_{j=1}^{r}\beta_{k_{j}}^{\left(N_{j}-1\right)}\Delta_{k_{j}}^{N_{j}}\right)
 f\left(x+\sum_{j=1}^{r}k_{j}y_{j}\right).\nonumber 
\end{align}
\end{thm}
\begin{proof}
Start with Theorem \ref{Thm7}
\[
\sum_{n_{1}=0}^{b^{N_{1}}-1}\xi^{s_{b}\left(n_{1}\right)}f\left(x+n_{1}y_{1}\right)=\left(-1\right)^{N_{1}}\sum_{k_{1}=0}^{b^{N_{1}}-N_{1}-1}\beta_{k_{1}}^{\left(N_{1}-1\right)}\Delta^{N_{1}}_{k_{1}}f\left(x+k_{1}y_{1}\right),
\]
replace $x$ with $x+n_{2}y_{2}$ on both sides, multiply by $\xi^{s_{b}\left(n_{2}\right)}$,
and sum over $0\le n_{2}\le b^{N_{2}-1}$ to obtain
\begin{align*}
\sum_{n_{1}=0}^{b^{N_{1}}-1}\sum_{n_{2}=0}^{b^{N_{2}}-1}\xi^{s_{b}\left(n_{1}\right)+s_{2}\left(n_{2}\right)}f\left(x+n_{1}y_{1}+n_{2}y_{2}\right) & =\left(-1\right)^{N_{1}}\sum_{k_{1}=0}^{b^{N_{1}}-N_{1}-1}\sum_{n_{2}=0}^{b^{N_{2}}-1}\xi^{s_{b}\left(n_{2}\right)}\beta_{k_{1}}^{\left(N_{1}-1\right)}\Delta^{N_{1}}_{k_{1}}f\left(x+k_{1}y_{1}+n_{2}y_{2}\right)\\
 & =\left(-1\right)^{N_{1}}\sum_{k_{1}=0}^{b^{N_{1}}-N_{1}-1}\beta_{k_{1}}^{\left(N_{1}-1\right)}\Delta^{N_{1}}_{k_{1}}\left[\sum_{n_{2}=0}^{b^{N_{2}}-1}\xi^{s_{b}\left(n_{2}\right)}f\left(x+k_{1}y_{1}+n_{2}y_{2}\right)\right].
\end{align*}
The inner sum can be computed by Thm. \ref{Thm7} as
\[
\sum_{n_{2}=0}^{b^{N_{2}}-1}\left(-1\right)^{s_{b}\left(n_{2}\right)}f\left(x+k_{1}y_{1}+n_{2}y_{2}\right)=\left(-1\right)^{N_{2}}\sum_{k_{2}=0}^{b^{N_{2}}-N_{2}-1}\beta_{k_{2}}^{\left(N_{2}-1\right)}\Delta^{N_{2}}_{k_{2}}f\left(x+k_{1}y_{1}+k_{2}y_{2}\right),
\]
so that
\begin{align*}
\sum_{n_{1}=0}^{b^{N_{1}}-1}\sum_{n_{2}=0}^{b^{N_{2}}-1}\xi^{s_{b}\left(n_{1}\right)+s_{b}\left(n_{2}\right)}f\left(x+n_{1}y_{1}+n_{2}y_{2}\right) & =\left(-1\right)^{N_{1}+N_{2}}\sum_{k_{1}=0}^{b^{N_{1}}-N_{1}-1}\sum_{k_{2}=0}^{b^{N_{2}}-N_{2}-1}\beta_{k_{1}}^{\left(N_{1}-1\right)}\beta_{k_{2}}^{\left(N_{2}-1\right)}\\
 & \times \Delta^{N_{1}}_{k_{1}}\Delta^{N_{2}}_{k_{2}}f\left(x+k_{1}y_{1}+k_{2}y_{2}\right).
\end{align*}
Repeating this operation $r-1$ times gives the result.
\end{proof}

\section{Conjectures}
\subsection{Conjecture 1}
We now use result \eqref{eq:tensor} to prove 
Conjecture \ref{conj42}.
\begin{prop}
Conjecture \ref{conj42} is true. In fact, for arbitrary base $b$ we have 
\[
\sum_{n_{1}=0}^{b^{N_{1}}-1}\dots\sum_{n_{r}=0}^{b^{N_{r}}-1}\xi^{\sum_{j=1}^{r}s_{b}\left(n_{j}\right)}\left(x+\sum_{j=1}^{r}n_{j}y_{j}\right)^{\sum_{j=1}^{r}N_{j}}=\left( \frac{1}{\xi-1}\right)^{\sum_{j=1}^{r}N_{r}}b^{\sum_{j=1}^{r}\frac{N_{j}\left(N_{j}+1\right)}{2}}\left(\prod_{j=1}^{r}y_{j}^{N_{j}}\right)\left(\sum_{j=1}^{r}N_{j}\right)!.
\]
\end{prop}
\begin{proof}
Consider Theorem \ref{thm2}: we note that each $\Delta_{k_{j}}$ operator 
reduces the degree of the polynomial $f$ by $1,$ so choosing
\[
f\left(x\right)=x^{\sum_{j=1}^{r}N_{j}}
\]
gives a 
right-hand side that does not depend on $x$. Under this choice of $f$, we obtain
\begin{align*}
\sum_{n_{1}=0}^{b^{N_{1}}-1}\dots\sum_{n_{r}=0}^{b^{N_{r}}-1} &\xi^{\sum_{j=1}^{r}s_{b}\left(n_{j}\right)}f\left(x+\sum_{j=1}^{r}n_{j}y_{j}\right) =\\
&\left(\sum_{j=1}^{r}N_{j}\right)!\left(-1\right)^{\sum_{j=1}^{r}N_{j}}\sum_{k_{1}=0}^{b^{N_{1}}-N_{1}-1}\dots\sum_{k_{r}=0}^{b^{N_{r}}-N_{r}-1}\left(\prod_{j=1}^{r}\beta_{k_{j}}^{\left(N_{j}-1\right)}\right)\prod_{j=1}^{r}y_{j}^{N_{j}}.
\end{align*}
In Lemma \ref{lemma6} we computed the simple sum
$$
\sum_{k=0}^{b^{N}-N-1} \beta_{k}^{\left( N -1 \right)} = \frac{b^{\frac{N\left( N+1 \right)}{2}}}{\left( 1-\xi \right)^N}.
$$
Therefore, we obtain the $r-$fold sum
\begin{align*}
\sum_{k_{1}=0}^{b^{N_{1}}-N_{1}-1}\dots\sum_{k_{r}=0}^{b^{N_{r}}-N_{r}-1}\prod_{j=1}^{r}\beta_{k_{j}}^{\left(N_{j}-1\right)} & =\prod_{j=1}^{r}\sum_{k_{j}=0}^{b^{N_{j}}-N_{j}-1}\beta_{k_{j}}^{\left(N_{j}-1\right)}
 =\prod_{j=1}^{r} \frac{b^{\frac{N_j\left( N_j+1 \right)}{2}}}{\left( 1-\xi \right)^{N_j}},
\end{align*}
which completes the proof.
\end{proof}

\subsection{Conjecture 2}
We now attack the second conjecture, using different methods. In \cite{Ulas},  the closed form \eqref{eq:conjecture2} is conjectured for $b=2$, and  a partial solution is formulated for higher $b$. We now prove the following general case.

\begin{thm}\label{thm12}
For $r\ge1$ and $N_{1},\dots N_{r}$ positive integers, we have
\[
\sum_{n_{1}=0}^{b^{N_{1}}-1}\dots\sum_{n_{r}=0}^{b^{N_{r}}-1}\xi^{\sum_{j=1}^{r}s_{b}\left(n_{j}\right)}\left(\sum_{j=1}^{r}s_{b}\left( n_j \right)x_j + n_{j}y_{j}\right)^{\sum_{j=1}^{r}N_{j}}
=\left(\frac{b}{\xi-1}\right)^{\sum_{j=1}^{r}N_{r}}\left(\sum_{j=1}^{r}N_{j}\right)! \prod_{j=1}^r \prod_{i_j=0}^{N_j-1}(x_j+b^{i_j}y_j).
\]
As a consequence, Conjecture 2 is true.
\end{thm}
\begin{proof}
As before, we attack the case $r=1$ first, and then use this to derive the result for arbitrary $r$. As an intermediate step, we consider the more general sum
\begin{equation}
S_{N,l} := \sum_{n=0}^{b^N-1} \xi^{s_b(n)}\left(s_b(n)x+ny\right)^l.
\end{equation}
This is very different from the previous conjecture; the mixing of digit based $s_b(n)$ and linear $n$ terms means that we cannot state a general functional result as before; instead, the methods developed to address this sum do not appear to generalize easily. Our approach is to derive a recurrence for $S_{N,l}$, from which we can show it is zero for $l<N$. We can then use this to kill most of the terms in the recurrence for $S_{N,N}$, leading to a simple functional equation which we can explicitly solve.

We begin by splitting the domain of summation into $b$ blocks of size $b^{N-1}$, each of which corresponds to a possibility for the first digit of $n$. We then exploit the recurrence $s_b(n+ kb^{N-1}) =k +s_b(n),$ which holds for all $n< b^{N-1}$ since this corresponds to adjoining the digit $k$ to the front of a shorter digit string. Accordingly, we have
\begin{align*}
S_{N,l} &=  \sum_{n=0}^{b^N-1} \xi^{s_b(n)}\left(s_b(n)x+ny\right)^l \\
& = \sum_{n=0}^{b^{N-1}-1} \sum_{k=0}^{b-1}\xi^{s_b(n+kb^{N-1})}\left(s_b(n+kb^{N-1})x+(n+kb^{N-1})y\right)^l \\
& = \sum_{n=0}^{b^{N-1}-1} \sum_{k=0}^{b-1}\xi^{s_b(n)+k}\left(s_b(n)x+ny + k(x+b^{N-1}y)\right)^l.
\end{align*}
We want to write this as a linear combination of the sums $S_{N-1,l}$, so we are forced to expand the inner term using the binomial theorem. The rest of the argument is a straightforward interchange of the order of summation. For what follows, we use the sequence of constants 
\begin{equation}
a_l = \sum_{k=0}^{b-1} k^l \xi^k.
\end{equation}
Note the special values $a_0 = 0$ and $a_1 = \frac{b}{\xi-1}$. 
Proceeding in this way, we obtain 
\begin{align*}
S_{N,l}& = \sum_{n=0}^{b^{N-1}-1} \sum_{k=0}^{b-1}\xi^{s_b(n)+k}\left(s_b(n)x+ny + k(x+b^{N-1}y)\right)^l \\
&= \sum_{n=0}^{b^{N-1}-1} \sum_{k=0}^{b-1} \xi^{s_b(n)+k} \sum_{m=0}^l \binom{l}{m}\left(s_b(n)x+ny\right)^m \left(x+b^{N-1}y\right)^{l-m} k^{l-m}\\
&= \sum_{m=0}^l \binom{l}{m}\left(x+b^{N-1}y\right)^{l-m} \sum_{n=0}^{b^{N-1}-1}\xi^{s_b(n)}\left(s_b(n)x+ny\right)^m \sum_{k=0}^{b-1} \xi^{k}  k^{l-m} \\
&=\sum_{m=0}^l \binom{l}{m}\left(x+b^{N-1}y\right)^{l-m} a_{l-m}S_{N-1,m}.
\end{align*}
Hence, we have the recurrence
\begin{equation}\label{sRecurrence}
S_{N,l}=\sum_{m=0}^{l-1} \binom{l}{m}\left(x+b^{N-1}y\right)^{l-m} a_{l-m}S_{N-1,m}. 
\end{equation}
In the last step, we used the fact that $a_0=0$ to get rid of the $S_{N-1,N}$ term. We can use this recurrence to derive a closed form for $S_{N,l}$. We begin with the base cases
\begin{equation}
S_{N,0} =  \begin{cases}1, & N=0, \\ 0, & N>0,\end{cases}
\end{equation}
the first of which is trivial and the second of which follows from the fact that $S_{N,0} = \sum_{n=0}^{b^N-1}\xi^{s_b(n)}$ sums over all possible strings of length $N$ with $b$ digits, so that we can average the parity of the digit sum of all of these strings.
We can now inductively show, for $N>0$, that $S_{N,l} = 0$ for $l<N$. We have just verified the base case, $l=0$, and the recurrence \eqref{sRecurrence} shows that we only sum over $S_{N-1,m}$ for $m \leq l-1 < N-1 $, so we can apply the inductive hypothesis. Hence, we now have the intermediate result
\begin{equation}
\label{SNl=0}
S_{N,l} = 0, \thinspace \thinspace l< N.
\end{equation}
We now apply the result \eqref{sRecurrence} while eliminating everything except the $m=N-1$ term, yielding
\begin{equation*}
S_{N,N} = \sum_{m=0}^{N-1} \binom{N}{m}\left(x+b^{N-1}y\right)^{N-m} a_{N-m}S_{N-1,m} = N\left(x+b^{N-1}y\right)a_1S_{N-1,N-1}.
\end{equation*}
Since we can explicitly compute $a_1 = \frac{b}{\xi-1}$, iterating this functional equation gives us the final closed form
\begin{equation}
S_{N,N} = \frac{b^N N!}{(\xi-1)^N} \prod_{l=0}^{N-1} \left(x+b^{l}y\right).
\end{equation}
Given this closed form (the case $r=1$ of our theorem), it is straightforward to derive the  case of an arbitrary value of $r$ by induction on $r.$ We only explicitly show the mechanism that allows us to skip from $r=1$ to $r=2$. Start from 
\begin{align*}
\sum_{n_{1}=0}^{b^{N_{1}}-1}&\sum_{n_{2}=0}^{b^{N_{2}}-1}\xi^{s_{b}\left(n_{1}\right)+s_{b}\left(n_{2}\right)}\left(s_{b}\left(n_{1}\right)x_{1}+n_{1}y_{1}+s_{b}\left(n_{2}\right)x_{2}+n_{2}y_{2}\right)^{N_{1}+N_{2}}\\
&=\sum_{n_{1}=0}^{b^{N_{1}}-1}\sum_{n_{2}=0}^{b^{N_{2}}-1}\xi^{s_{b}\left(n_{1}\right)+s_{b}\left(n_{2}\right)}\sum_{p=0}^{N_{1}+N_{2}}\binom{N_{1}+N_{2}}{p}\left(s_{b}\left(n_{1}\right)x_{1}+n_{1}y_{1}\right)^{p}\left(s_{b}\left(n_{2}\right)x_{2}+n_{2}y_{2}\right)^{N_{1}+N_{2}-p}\\
&=\sum_{p=0}^{N_{1}+N_{2}}\binom{N_{1}+N_{2}}{p}\sum_{n_{1}=0}^{b^{N_{1}}-1}\xi^{s_{b}\left(n_{2}\right)}\left(s_{b}\left(n_{1}\right)x_{1}+n_{1}y_{1}\right)^{p}\sum_{n_{2}=0}^{b^{N_{2}}-1}\xi^{s_{b}\left(n_{2}\right)}\left(s_{b}\left(n_{2}\right)x_{2}+n_{2}y_{2}\right)^{N_{1}+N_{2}-p}.
\end{align*}
Then apply result \eqref{SNl=0} so that only the $p=N_{1}$ term remains and
we obtain
\[
\binom{N_{1}+N_{2}}{N_{1}}\sum_{n_{1}=0}^{b^{N_{1}}-1}\xi^{s_{b}\left(n_{2}\right)}\left(s_{b}\left(n_{1}\right)x_{1}+n_{1}y_{1}\right)^{N_{1}}\sum_{n_{2}=0}^{b^{N_{2}}-1}\xi^{s_{b}\left(n_{2}\right)}\left(s_{b}\left(n_{2}\right)x_{2}+n_{2}y_{2}\right)^{N_{2}},
\]
which is then equal to
\begin{align*}
\frac{\left(N_{1}+N_{2}\right)!}{N_{1}!N_{2}!} \times\frac{b^{N_1} N_1!}{(\xi-1)^{N_1}}\prod_{i_{1}=0}^{N_{1}-1}\left(x_{1}+b^{i_{1}}y_{1}\right) \times \frac{b^{N_2} N_2!}{(\xi-1)^{N_2}}\prod_{i_{2}=0}^{N_{2}-1}\left(x_{2}+b^{i_{2}}y_{2}\right)\\
=\left(N_{1}+N_{2}\right)!\left(\frac{b}{\xi-1}\right)^{N_{1}+N_{2}}\prod_{j=1}^{2}\prod_{i_{j}=0}^{N_{j}-1}\left(x_{j}+b^{i_{j}}y_{j}\right).
\end{align*}
Iterating these steps yields the claimed result.
\end{proof}

\subsection{A family of polynomials}
Byszewski and Ulas also introduced the much more complicated family of polynomials $H_{m,N}$ as defined in \eqref{eq:pb3}. They then conjectured the explicit form \cite[Remark 4.6]{Ulas}
\[
H_{2,N}\left( t,\mathbf{x} \right)=\left(-1\right)^{N}N!2^{\frac{N\left(N-1\right)}{2}}\left(2\frac{x_{1}^{N}-x_{2}^{N}}{x_{1}-x_{2}}t+2^{N}\frac{x_{1}^{N+1}-x_{2}^{N+1}}{x_{1}-x_{2}}+x_{1}x_{2}\left(2^{N}-1\right)\frac{x_{1}^{N-1}-x_{2}^{N-1}}{x_{1}-x_{2}}\right).
\]
Note that this contains the mixing term $s_2\left(\sum_{j=1}^m i_j\right)$, whereas in Conjecture \ref{conj42} each $s_2(n_j)$ acted independently. Using (\ref{eq:tensor}), we prove the following result.
\begin{thm}
The identity
\begin{equation}
H_{2,N}\left( t,\mathbf{x} \right)=\left(-1\right)^{N}N!2^{\frac{N\left(N-1\right)}{2}}\left(2\frac{x_{1}^{N}-x_{2}^{N}}{x_{1}-x_{2}}t+2^{N}\frac{x_{1}^{N+1}-x_{2}^{N+1}}{x_{1}-x_{2}}+x_{1}x_{2}\left(2^{N}-1\right)\frac{x_{1}^{N-1}-x_{2}^{N-1}}{x_{1}-x_{2}}\right)\label{eq:H2n}
\end{equation}
holds.
\end{thm}
\begin{proof}
We start with the more general form
\[
H_{m,N,p}\left( t,\mathbf{x} \right)=\sum_{i_{1}=0}^{2^{N}-1}\dots\sum_{i_{m}=0}^{2^{N}-1}\left(-1\right)^{s_{2}\left(\sum_{j=1}^{m}i_{j}\right)}\left(t+\sum_{j=1}^{m}i_{j}x_{j}\right)^{p},
\]
and compute the special case
\begin{align*}
H_{2,N,p}\left( t,\mathbf{x} \right) & =\sum_{i_{1}=0}^{2^{N}-1}\sum_{i_{2}=0}^{2^{N}-1}\left(-1\right)^{s_{2}\left(i_{1}+i_{2}\right)}\left(t+i_{1}x_{1}+i_{2}x_{2}\right)^{p}\\
 & =\sum_{n=0}^{2^{N}-1}\left(-1\right)^{s_{2}\left(n\right)}\sum_{i_{2}=0}^{n}\left(t+\left(n-i_{2}\right)x_{1}+i_{2}x_{2}\right)^{p}\\
 & +\sum_{n=2^{N}}^{2^{N+1}-2}\left(-1\right)^{s_{2}\left(n\right)}\sum_{i_{2}=n-2^{N}+1}^{2^{N}-1}\left(t+\left(n-i_{2}\right)x_{1}+i_{2}x_{2}\right)^{p}.
\end{align*}
We use Faulhaber's sum of powers formula 
\[
\sum_{i=r}^{s-1}\left(a+bi\right)^{p}=\frac{b^{p}}{p+1}\left[B_{p+1}\left(\frac{a}{b}+s\right)-B_{p+1}\left(\frac{a}{b}+r\right)\right],
\]
where $B_{p}\left(x\right)$ is the Bernoulli polynomial with generating
function
\[
\sum_{n\ge0}\frac{B_{n}\left(x\right)}{n!}z^{n}=\frac{z}{e^{z}-1}e^{zx}.
\]
We deduce the first inner sum as
\begin{align*}
\sum_{i_{2}=0}^{n}\left(t+\left(n-i_{2}\right)x_{1}+i_{2}x_{2}\right)^{p} & =\sum_{i_{2}=0}^{n}\left(t+nx_{1}+\left(x_{2}-x_{1}\right)i_{2}\right)^{p}\\
 & =\frac{\left(x_{2}-x_{1}\right)^{p}}{p+1}\left[B_{p+1}\left(\frac{t+nx_{1}}{x_{2}-x_{1}}+n+1\right)-B_{p+1}\left(\frac{t+nx_{1}}{x_{2}-x_{1}}\right)\right]\\
 & =\frac{\left(x_{2}-x_{1}\right)^{p}}{p+1}\left[B_{p+1}\left(\frac{t+nx_{2}}{x_{2}-x_{1}}+1\right)-B_{p+1}\left(\frac{t+nx_{1}}{x_{2}-x_{1}}\right)\right],
\end{align*}
and the second inner sum as
\begin{align*}
\sum_{i_{2}=n-2^{N}+1}^{2^{N}-1}&\left(t+\left(n-i_{2}\right)x_{1}+i_{2}x_{2}\right)^{p}  =\frac{\left(x_{2}-x_{1}\right)^{p}}{p+1}\left[B_{p+1}\left(\frac{t+nx_{1}}{x_{2}-x_{1}}+2^{N}\right)-B_{p+1}\left(\frac{t+nx_{1}}{x_{2}-x_{1}}+n-2^{N}+1\right)\right]\\
 & =\frac{\left(x_{2}-x_{1}\right)^{p}}{p+1}\left[B_{p+1}\left(\frac{t+nx_{1}}{x_{2}-x_{1}}+2^{N}\right)-B_{p+1}\left(\frac{t+nx_{2}}{x_{2}-x_{1}}-2^{N}+1\right)\right].
\end{align*}
We deduce
\begin{align*}
H_{2,N,p} \left( t,\mathbf{x} \right)& =\frac{\left(x_{2}-x_{1}\right)^{p}}{p+1}\left[\sum_{n=0}^{2^{N}-1}\left(-1\right)^{s_{2}\left(n\right)}\left[B_{p+1}\left(\frac{t+nx_{2}}{x_{2}-x_{1}}+1\right)-B_{p+1}\left(\frac{t+nx_{1}}{x_{2}-x_{1}}\right)\right]\right]\\
 & +\frac{\left(x_{2}-x_{1}\right)^{p}}{p+1}\left[\sum_{n=2^{N}}^{2^{N+1}-1}\left(-1\right)^{s_{2}\left(n\right)}\left[B_{p+1}\left(\frac{t+nx_{1}}{x_{2}-x_{1}}+2^{N}\right)-B_{p+1}\left(\frac{t+nx_{2}}{x_{2}-x_{1}}-2^{N}+1\right)\right]\right]\\
 & -\frac{\left(x_{2}-x_{1}\right)^{p}}{p+1}\left[B_{p+1}\left(\frac{t+\left(2^{N+1}-1\right)x_{1}}{x_{2}-x_{1}}+2^{N}\right)-B_{p+1}\left(\frac{t+\left(2^{N+1}-1\right)x_{2}}{x_{2}-x_{1}}-2^{N}+1\right)\right],
\end{align*}
where the last term corresponds to $n=2^{N+1}-1$ and is easily seen
to be equal to $0.$ Reindexing the second sum yields
\begin{align*}
\sum_{n=2^{N}}^{2^{N+1}-1}&\left(-1\right)^{s_{2}\left(n\right)}\left[B_{p+1}\left(\frac{t+nx_{1}}{x_{2}-x_{1}}+2^{N}\right)-B_{p+1}\left(\frac{t+nx_{2}}{x_{2}-x_{1}}-2^{N}+1\right)\right]\\
&=\sum_{n=0}^{2^{N}-1}\left(-1\right)^{s_{2}\left(n+2^{N}\right)}\left[B_{p+1}\left(\frac{t+\left(n+2^{N}\right)x_{1}}{x_{2}-x_{1}}+2^{N}\right)-B_{p+1}\left(\frac{t+\left(n+2^{N}\right)x_{2}}{x_{2}-x_{1}}-2^{N}+1\right)\right]\\
&=-\sum_{n=0}^{2^{N}-1}\left(-1\right)^{s_{2}\left(n\right)}\left[B_{p+1}\left(\frac{t+nx_{1}+2^{N}x_{2}}{x_{2}-x_{1}}\right)-B_{p+1}\left(\frac{t+nx_{2}+2^{N}x_{1}}{x_{2}-x_{1}}+1\right)\right],
\end{align*}
where we have used the fact that
\[
\left(-1\right)^{s_{2}\left(n+2^{N}\right)}=\left(-1\right)^{s_{2}\left(n\right)+1},\thinspace\thinspace0\le n\le2^{N}-1.
\]
Hence we have
\begin{align*}
&H_{2,N,p}\left( t,\mathbf{x} \right)  =\frac{\left(x_{2}-x_{1}\right)^{p}}{p+1}\sum_{n=0}^{2^{N}-1}\left(-1\right)^{s_{2}\left(n\right)}\left[B_{p+1}\left(\frac{t+nx_{2}}{x_{2}-x_{1}}+1\right)-B_{p+1}\left(\frac{t+nx_{1}}{x_{2}-x_{1}}\right)\right]\\
 & -\frac{\left(x_{2}-x_{1}\right)^{p}}{p+1}\sum_{n=0}^{2^{N}-1}\left(-1\right)^{s_{2}\left(n\right)}\left[B_{p+1}\left(\frac{t+nx_{1}+2^{N}x_{2}}{x_{2}-x_{1}}\right)-B_{p+1}\left(\frac{t+nx_{2}+2^{N}x_{1}}{x_{2}-x_{1}}+1\right)\right].
\end{align*}
Now we apply the finite difference identity (\ref{eq:tensor}) to obtain
\begin{align*}
H_{2,N,p}\left( t,\mathbf{x} \right) & =\left(-1\right)^{N}\frac{\left(x_{2}-x_{1}\right)^{p}}{p+1}\sum_{k=0}^{2^{N}-N-1}\alpha_{k}^{\left(N-1\right)}\Delta_{k}^{N}\\
 & \times \left[B_{p+1}\left(\frac{t+kx_{2}}{x_{2}-x_{1}}+1\right)-B_{p+1}\left(\frac{t+kx_{1}}{x_{2}-x_{1}}\right)
 -B_{p+1}\left(\frac{t+kx_{1}+2^{N}x_{2}}{x_{2}-x_{1}}\right)+B_{p+1}\left(\frac{t+kx_{2}+2^{N}x_{1}}{x_{2}-x_{1}}+1\right)\right].
\end{align*}
The action of $\Delta_{k}^{N}$ on $B_{p+1}$, which is a polynomial
of degree $p+1$, gives a polynomial of degree $p+1-N.$ Therefore, we also have the easy corollary
$$H_{m,N,p}\left( t,\mathbf{x} \right) = 0, \thinspace\thinspace 0\leq p \leq N-2.$$
 Let us choose
$N=p$ so we end up with a polynomial of degree $1$. In fact, we have
\[
\Delta_{k}^{N}B_{N+1}\left(a+kb\right)=\left(N+1\right)!b^{N}\left(a+kb+\frac{bN-1}{2}\right).
\]
This computation is based on the integral representation
\begin{equation}
\Delta_{k}f\left(x+ky\right)=y\int_{0}^{1}f'\left(x+\left(k+u\right)y\right)du.\label{eq:integral representation}
\end{equation}
Upon iterating the integral (\ref{eq:integral representation}),
we deduce
\[
\Delta_{k}^{N}B_{N+1}\left(a+kb\right)=b^{N}\int_{0}^{1}\dots\int_{0}^{1}B_{N+1}^{\left(N\right)}\left(a+\left(k+u_{1}+\dots+u_{N}\right)b\right)du_{1}\dots du_{N}.
\]
Since the Bernoulli polynomials are of the binomial type, their derivative
is
\[
B_{N+1}^{\left(N\right)}\left(z\right)=\left(N+1\right)!B_{1}\left(z\right)=\left(N+1\right)!\left(z-\frac{1}{2}\right).
\]
Combining this fact with the simple integral
\begin{align*}
\int_{0}^{1}\dots\int_{0}^{1}\left(a+\left(k+u_{1}+\dots+u_{N}\right)b-\frac{1}{2}\right)du_{1}\dots du_{N}
 =a-\frac{1}{2}+b\left(k+\frac{N}{2}\right)
\end{align*}
yields the result.

Applying this finite difference result, we have
\begin{align*}
H_{2,N,p}\left(t,\mathbf{x}\right) & =\left(-1\right)^{N}\left(N+1\right)!\frac{\left(x_{2}-x_{1}\right)^{p}}{p+1}\sum_{k=0}^{2^{N}-N-1}\alpha_{k}^{\left(N-1\right)}\\
 & \times \left[\left(\frac{x_{2}}{x_{2}-x_{1}}\right)^{N}\left(\frac{t+kx_{2}+\frac{N}{2}x_{2}}{x_{2}-x_{1}}+\frac{1}{2}\right)-\left(\frac{x_{1}}{x_{2}-x_{1}}\right)^{N}\left(\frac{t+kx_{1}+\frac{N}{2}x_{1}}{x_{2}-x_{1}}-\frac{1}{2}\right)\right.\\
 & \left.-\left(\frac{x_{1}}{x_{2}-x_{1}}\right)^{N}\left(\frac{t+2^{N}x_{2}+kx_{1}+\frac{N}{2}x_{1}}{x_{2}-x_{1}}-\frac{1}{2}\right)+\left(\frac{x_{2}}{x_{2}-x_{1}}\right)^{N}\left(\frac{t+2^{N}x_{1}+kx_{2}+\frac{N}{2}x_{2}}{x_{2}-x_{1}}+\frac{1}{2}\right)\right]
\end{align*}
We can then obtain the zero order moment
\[
\sum_{k=0}^{2^{N}-N-1}\alpha_{k}^{\left(N-1\right)}=2^{\frac{N\left(N-1\right)}{2}}
\]
by setting $z=1$ in \eqref{eq:gf} - or $b=2$ in \eqref{sumbetak}, and the first moment
\[
\sum_{k=0}^{2^{N}-N-1}k\alpha_{k}^{\left(N-1\right)}=2^{\frac{N\left(N-1\right)}{2}}\left(2^{N-1}-\frac{N+1}{2}\right)
\]
by taking a logarithmic derivative of \eqref{eq:gf} and setting $z=1$, or by setting $b=2$ in \eqref{sumkbetak}.
Comparing coefficients of $t$ in $H_{2,N,p}$ then yields
\[
\left(-1\right)^{N}N!2^{\frac{N\left(N-1\right)}{2}+1}\frac{x_{2}^{N}-x_{1}^{N}}{x_{2}-x_{1}},
\]
while simplifying the constant coefficient gives
\begin{align*}\left(-1\right)^{N}N!2^{\frac{N\left(N-1\right)}{2}}\left[x_{2}^{N}\left(\frac{1}{2}+\frac{Nx_{2}+2^{N}x_{1}}{x_{2}-x_{1}}+\frac{1}{2}\right)\right]
-\left(-1\right)^{N}N!2^{\frac{N\left(N-1\right)}{2}}\left[x_{1}^{N}\left(-\frac{1}{2}+\frac{Nx_{1}+2^{N}x_{2}}{x_{2}-x_{1}}-\frac{1}{2}\right)\right]\\
+\left(-1\right)^{N}N!2^{\frac{N\left(N-1\right)}{2}}\left(2^{N-1}-\frac{N-1}{2}-1\right)\left[\frac{2}{x_{2}-x_{1}}\left(x_{2}^{N+1}-x_{1}^{N+1}\right)\right].
\end{align*}
After some algebra, this can be shown to coincide with the constant
term in (\ref{eq:H2n}), which completes the proof.
\end{proof}

We conclude with the general base $b$ case.

\begin{thm}
Let us denote $X_{1}=\frac{x_1}{x_2 - x_1}$ and $X_{2}=\frac{x_2}{x_2 - x_1}$, and
\[
m_{0}=\sum_{k=0}^{b^{N}-N-1}\beta_{k}^{\left(N-1\right)},\thinspace\thinspace m_{1}=\sum_{k=0}^{b^{N}-N-1}k\beta_{k}^{\left(N-1\right)},
\]
as computed in Lemma \ref{lemma6}.
The extension to an arbitrary base $b$ gives
\begin{align*}
\sum_{i_{1}=0}^{b^{N}-1}\sum_{i_{2}=0}^{b^{N}-1}\xi^{s_{b}\left(i_{1}+i_{2}\right)}\left(t+i_{1}x_{1}+i_{2}x_{2}\right)^{N} & =ct+d
\end{align*}
with
\[
\frac{\left(-1\right)^{N}}{N!} c= m_{0}\left(1-\xi\right)\frac{x_{2}^{N}-x_{1}^{N}}{x_{2}-x_{1}}.
\]

and
\begin{align*}
&\frac{\left(-1\right)^{N}}{N!} d =\left[x_{2}^{N}\left(m_{0}\left(\frac{1}{2}+\frac{N}{2}
\right)
X_2+m_{1}
\right)
X_2
-x_{1}^{N}\left(m_{0}\left(\frac{N}{2}
X_1-\frac{1}{2}\right)+m_{1}
X_1\right)\right]\\
+&\xi\left[x_{1}^{N}\left(m_{0}\left(b^{N}
X_2
+\frac{N}{2}
X_1
-\frac{1}{2}\right)+m_{1}
X_1 \right)-\xi x\left(m_{0}\left(b^{N}
X_1
+\frac{1}{2}+\frac{N}{2}
X_2
\right)+m_{1}
X_2
\right)\right].
\end{align*}
\end{thm}

\begin{proof}
The extension to an arbitrary base $b$ gives
\begin{align*} & \sum_{i_{1}=0}^{b^{N}-1}\sum_{i_{2}=0}^{b^{N}-1}\xi^{s_{b}\left(i_{1}+i_{2}\right)}\left(t+i_{1}x_{1}+i_{2}x_{2}\right)^{N}=\left(-1\right)^{N}\frac{\left(x_{2}-x_{1}\right)^{N}}{N+1}\sum_{k=0}^{b^{N}-N-1}\beta_{k}^{\left(N-1\right)}\left(N+1\right)!\\
 & \left[\left(\frac{x_{2}}{x_{2}-x_{1}}\right)^{N}\left(\frac{t+\left(k+\frac{N}{2}\right)x_{2}}{x_{2}-x_{1}}+\frac{1}{2}\right)\right.-\left(\frac{x_{1}}{x_{2}-x_{1}}\right)^{N}\left(\frac{t+\left(k+\frac{N}{2}\right)x_{1}}{x_{2}-x_{1}}-\frac{1}{2}\right)\\
 & +\xi\left(\frac{x_{1}}{x_{2}-x_{1}}\right)^{N}\left(\frac{t+b^{N}x_{2}+\left(k+\frac{N}{2}\right)x_{1}}{x_{2}-x_{1}}-\frac{1}{2}\right)
  \left.-\xi\left(\frac{x_{2}}{x_{2}-x_{1}}\right)^{N}\left(\frac{t+b^{N}x_{1}+\left(k+\frac{N}{2}\right)x_{2}}{x_{2}-x_{1}}+\frac{1}{2}\right)\right].
\end{align*}
We deduce
\begin{align*}
&\sum_{i_{1}=0}^{b^{N}-1}\sum_{i_{2}=0}^{b^{N}-1}  \xi^{s_{b}\left(i_{1}+i_{2}\right)}\left(t+i_{1}x_{1}+i_{2}x_{2}\right)^{N}=\\
 & \left(-1\right)^{N}N!\left[x_{2}^{N}\left(m_{0}\left(\frac{t+\frac{N}{2}x_{2}}{x_{2}-x_{1}}+\frac{1}{2}\right)+m_{1}\frac{x_{2}}{x_{2}-x_{1}}\right)-\left(-1\right)^{N}x_{1}^{N}\left(m_{0}\left(\frac{t+\frac{N}{2}x_{1}}{x_{2}-x_{1}}-\frac{1}{2}\right)+m_{1}\frac{x_{1}}{x_{2}-x_{1}}\right)\right.\\
 & \left.+\xi\left(-1\right)^{N}x_{1}^{N}\left(m_{0}\left(\frac{t+b^{N}x_{2}+\frac{N}{2}x_{1}}{x_{2}-x_{1}}-\frac{1}{2}\right)+m_{1}\frac{x_{1}}{x_{2}-x_{1}}\right)-\xi\left(-1\right)^{N}x_{2}^{N}\left(m_{0}\left(\frac{t+b^{N}x_{1}+\frac{N}{2}x_{2}}{x_{2}-x_{1}}+\frac{1}{2}\right)+m_{1}\frac{x_{2}}{x_{2}-x_{1}}\right)\right].
\end{align*}
Identifying the constant and first order terms in this expression gives the result.
\end{proof}
\section*{Acknowledgements}
None of the authors have any competing interests in the manuscript.
This material is based upon work supported by the National Science Foundation under Grant No. DMS-1439786 while the second author was in residence at the Institute for Computational and Experimental Research in Mathematics in Providence, RI, during the Point Configurations in Geometry, Physics and Computer Science Semester Program, Spring 2018.

\end{document}